\documentclass{amsart}
\usepackage{amsmath}
\usepackage{graphicx}
\usepackage{hyperref}
\usepackage{amssymb}
\hypersetup{colorlinks=true,citecolor=blue,linkcolor=cyan}
\newtheorem{thm}{Theorem}[section]
\newtheorem{cor}[thm]{Corollary}
\newtheorem{lem}[thm]{Lemma}
\newtheorem{prop}[thm]{Proposition}
\theoremstyle{definition}

\theoremstyle{remark}
\newtheorem{rem}[thm]{Remark}
\numberwithin{equation}{section}
\newcommand{\F}{\mathcal{F}}
\newcommand{\Ff}{\mathbb{F}}

\newcommand{\s}{\mathcal{S}}

\newcommand{\Z}{\mathbb{Z}}
\newcommand{\Hh}{\mathcal{H}}
\newcommand{\Tr}{\mbox{Tr}}

\newcommand{\supp}{\mbox{supp}}
\newcommand{\Pp}{\mathbb{P}}

\newcommand{\rad}{\mbox{rad}}
\newcommand{\Res}{\mbox{Res}}

\begin{document}

\title[Traces of Frobenius of Cyclic Covers]{Lower Order Terms for Expected Value of Traces of Frobenius of a Family of Cyclic Covers of $\Pp^1_{\Ff_q}$ and One-Level Densities}%
\author{Patrick Meisner}%
\address{Department of Mathematics \\ KTH Royal Institute of Technology}%
\email{pfmeisner@gmail.com}%
\thanks{The author was supported by the Verg foundation.}
%\thanks{}%
%\subjclass{}%
%\keywords{}%

%\date{}%
%\dedicatory{}%
%\commby{}%
% ----------------------------------------------------------------
\begin{abstract}

We consider the expected value of $\Tr(\Theta_C^n)$ where $C$ runs over a thin family of $r$-cyclic covers of $\Pp^1_{\Ff_q}$ for any $r$. We obtain many lower order terms dependent on the divisors of $r$. We use these results to calculate the one-level density of the family and hypothesize a refined one-level density result.

\end{abstract}
\maketitle
% ----------------------------------------------------------------
\section{Introduction}\label{intro}

Fix a prime $p$ and let $q=p^a$. Then for any curve, $C$, defined over $\Ff_q$, we can define the zeta function attached to the curve as
\begin{align}\label{zetafunc}
Z_C(u) = \exp\left( \sum_{n=1}^\infty \#C(\Ff_{q^n})\frac{u^n}{n}  \right).
\end{align}
The Riemann hypothesis (proved by Weil \cite{W}) states that
\begin{align}\label{Lfunc}
Z_C(u) = \frac{L_C(u)}{(1-u)(1-qu)}
\end{align}
where $L_C(u)$ is a polynomial of degree $2g$ (where $g$ is the genus of the curve), all of whose zeros lie on the ``half-line" $|u|=q^{-1/2}$. Moreover, the distribution of these zeroes give us arithmetic information about the number of points on the curve. Therefore, it is of interest to study statistical properties of these zeroes.

One such statistical object that is of interest is the one-level density of the zeroes. Specifically, for any Schwartz test function, $f$, define
\begin{align}\label{OLD1}
\mathcal{D}(L_C,f) = \sum_{\theta} f\left(2g \frac{\theta}{2\pi}\right)
\end{align}
where the sum is over all real numbers $\theta$ such that $q^{-1/2} e^{i\theta}$ are the zeros of $L_C(u)$, counted with multiplicity. Note that by the cyclic nature of $e^{i\theta}$ if $\theta$ appears in the sum then so does $\theta+2\pi n$ for all $n\in\mathbb{Z}$. Since $f$ is Schwartz, it's mass is concentrated near $0$, and so $\mathcal{D}(L_C,f)$ can be viewed as a measure on the zeroes near the real line, the so-called ``low lying zeroes".

Since $L_C(u)$ is a polynomial of degree $2g$ whose zeroes lie on $|u|=q^{-1/2}$, we get that we can find a unitary $2g\times2g$ matrix $\Theta_C$, called the Frobenius of the curve $C$, such that
\begin{align}\label{Frob}
L_C(u) = \det(1-\sqrt{q} u\Theta_C).
\end{align}
%In fact, using the functional equation of $L_C(u)$, one can actually show that $\Theta_C$ is a symplectic matrix.

Therefore, if for any unitary $2g\times2g$ matrix $U$ and even Schwartz function, $f$, we define
$$\mathcal{D}(U,f) = \sum_{j=1}^{2g} \sum_{n\in\Z} f\left(2g\left(\frac{\theta_j}{2\pi} -n\right)\right) $$
where the $\theta_j$ are the eigenvalues of $U$ then we get
\begin{align}\label{OLD2}
\mathcal{D}(L_C,f)= \mathcal{D}(\Theta_C,f)
\end{align}

Katz and Sarnak \cite{KS} predicted that for a nice family of curves $\Hh$ then there should be a matrix group $M$, called the monodromy group, such that
$$\langle \mathcal{D}(L_C,f)\rangle_{\Hh}\footnote{Here we use the notation for any set $S$ and any function $\phi$ on $S$ that $\langle\phi(s)\rangle_S =\frac{1}{|S|}\sum_{s\in S} \phi(s)$ } \sim \int_M \mathcal{D}(U,f) dU.$$
They also predicted that $M$ should be one of $USp$, the unitary symplectics; $O$, the orthogonals; $SO(even)$, the even special orthogonals; $SO(odd)$, the odd special orthogonal; or $U$, the unitaries.

These predictions have been confirmed for multiple families. Classically, over number fields Katz and Sarnak \cite{KS} showed that the family of $L$-functions attached to quadratic characters have symplectic monodromy group. In \cite{M2}, I showed that the cubic characters have unitary monodromy group whereas Cho and Park \cite{CP} proved a more general result for this family.  Shankar, S\"{o}dergren and Templier \cite{SST} showed that $L$-functions attached to $S_3,S_4$ and $S_5$ extensions of $\mathbb{Q}$ have symplectic monodromy group while Yang \cite{Y} showed that the $S_3$ and $S_4$  extensions of $\Ff_q[T]$ also have symplectic monodromy group.

Using Poisson summation one can show that for any $U$, we get
\begin{align}\label{D(u,f)}
\mathcal{D}(U,f) = \frac{1}{2g}\sum_{n\in\mathbb{Z}} \hat{f}\left(\frac{n}{2g}\right) \Tr(U^n).
\end{align}
%Further, if we suppose $U$ is symplectic and $f$ is an even function, then we can use the fact that $U^{-1} =-U$ for symplectic matrices to simplify this down to
%$$\mathcal{D}(U,f) = \hat{f}(0) + \frac{1}{g}\sum_{n=1}^{\infty} \hat{f}\left(\frac{n}{g}\right) \Tr(U^{2n}). $$
Therefore to understand the statistics of $\mathcal{D}(L_C,f)$, it is enough to understand the statistics of $\Tr(\Theta_C^n)$. Typically, one is able to calculate the expected value of $\Tr(\Theta_C^n)$ as long as $n<2\alpha g$ for some $\alpha$. Then restricting down to $f$ such that $\supp(\hat{f})\subset (-\alpha,\alpha)$ allows you to calculate the one level density.

\subsection{Cyclic Covers}\label{background}

In this paper we are interested in the expected value of $\Tr(\Theta_C^n)$ as $C$ varies over cyclic covers. Namely, the family
\begin{align}\label{curv}
\Hh_r = \{ (C,\pi) : \pi:C\to\Pp^1 \mbox{ is a degree-$r$ map defined over $\Ff_q$} \}
\end{align}
for any positive integer $r$.

Clearly $\Hh_r$ is an infinite set. Therefore, we wish to find an invariant such that only finitely many curves have a given invariant. The obvious choice is the genus as we want to associate the family of $\Theta_C$ to a matrix group and we have that if the genus is fixed then the dimension of $\Theta_C$ is fixed. Hence, define the (finite) set of curves
\begin{align}\label{curvgen}
\Hh^{gen}_r(g) = \{(C,\pi)\in\Hh_r : g(C)=g\}
\end{align}
where we denote $g(C)$ as the genus of $C$.

This was first considered in the case $r=2$ by Rudnick \cite{R} in the affine setting and then extended by Chinis \cite{C} to the projective setting.

\begin{thm}[\cite{C},\cite{R}]

Let $q$ be odd. If $n$ is odd, then $\langle \Tr(\Theta_C^n)\rangle_{\Hh^{gen}_2(g)} =0$. Otherwise, if $n$ is even, then
$$\langle \Tr(\Theta_C^n)\rangle_{\Hh^{gen}_2(g)} = \begin{cases} -1 & 0< n<2g \\ -1-\frac{1}{q-1} & n=2g \\  0 & n>2g\end{cases} -D^{*}_2(g,n) + O\left(nq^{n/2-2g} + gq^{-g} \right) $$
where $D^{*}_2(g,n)$ is some explicit function such that $D^{*}_2(g,n)\ll nq^{-n/2}$
\end{thm}

Now, we notice that the first part resembles that of the expected value of traces of $USp(2g)$. Namely, that
$$\int_{USp(2g)} \Tr(U^n) dU = \begin{cases} 2g & n=0 \\ -1 & 0<n\leq 2g, n \mbox{ even} \\ 0 & \mbox{ otherwise} \end{cases}.$$
Therefore, one may express the above result and then a one-level density result in terms of the matrix group $USp(2g)$.

\begin{cor}[\cite{C},\cite{R}]
If $3\log_qg<n<4g-5\log_qg$, but $n\not=2g$, then
$$\langle \Tr(\Theta_C^n)\rangle_{\Hh^{gen}_2(g)} = \int_{USp(2g)} \Tr(U^n) dU + O\left(\frac{1}{g}\right).$$
Hence, if we also incorporate $D_2(g,n)$, we get that for any even Schwartz test function $f$ with $\supp(\hat{f}) \subset (-2,2)$, then
$$\langle\mathcal{D}(L_C,f)\rangle_{\Hh^{gen}_2(g)} = \int_{USp(2g)} \mathcal{D}(U,f) dU + \frac{dev^{*}_2(f)}{g} + o\left(\frac{1}{g}\right)$$
where
$$dev^{*}_2(f) = \hat{f}(0)\sum_{P}\frac{\deg(P)}{|P|^2-1} -\hat{f}(1)\frac{1}{q-1}$$
and the sum is over all prime polynomials with $|P|=q^{\deg(P)}$.
\end{cor}

Subsequently, Bucur et al. \cite{BCDG+} consider the case where $r=\ell$ is an odd prime. As it turns out $\Hh^{gen}_\ell(g) \not=\emptyset$ if and only if $2g\equiv 0 \mod{\ell-1}$. Hence, we restrict only to those possible values of $g$.

\begin{thm}[\cite{BCDG+}]\label{BCDthm}
Let $\ell$ be an odd prime and suppose $q\equiv 1 \mod{\ell}$ and $2g\equiv 0 \mod{\ell-1}$. Then
$$\langle \Tr(\Theta_C^n)\rangle_{\Hh^{gen}_\ell(g)} =  -D^{*}_\ell(g,n) + O\left(\frac{n^{\ell-2}q^{n/\ell}}{gq^{n/2}} + q^{n(\frac{1}{2}+\epsilon) - (\frac{1}{2}-\epsilon)\frac{2g}{\ell-1}}\right) $$
where $D^{*}_\ell(g,n)$ is some explicit function such that $D^{*}_\ell(g,n)\ll nq^{-n/2}$
\end{thm}

Thus we see in this case, the main term is actually $0$ and this resembles the expected value of traces of $U(2g)$. Namely, that
$$\int_{U(2g)} \Tr(U^n)=\begin{cases} 2g & n=0 \\ 0 & n\not=0\end{cases}.$$
Therefore, one may express the above result and then a one-level density result in terms of the matrix group $U(2g)$.

\begin{cor}[\cite{BCDG+}]
Let $\ell$ be an odd prime, $q\equiv 1\mod{\ell}$ and $2g\equiv 0 \mod{\ell-1}$. Then, for any $\epsilon>0$ and any $n$ such that $6\log_qg < n< (1-\epsilon) \left(\frac{2g}{\ell-1}+2\right)$, as $g\to\infty$,
$$\langle \Tr(\Theta_C^n)\rangle_{\Hh^{gen}_\ell(g)} = \int_{U(2g)} \Tr(U^n) dU + O\left(\frac{1}{g}\right).$$
Hence, if we also incorporate $D^{*}_\ell(g,n)$, we get that for any even Schwartz test function $f$ with $\supp(\hat{f}) \subset (-\frac{1}{\ell-1},\frac{1}{\ell-1})$, then
$$\langle\mathcal{D}(L_C,f)\rangle_{\Hh^{gen}_\ell(g)} = \int_{U(2g)} \mathcal{D}(U,f) dU + \frac{dev^{*}_\ell(f)}{g} + O\left(\frac{1}{g^{2-\epsilon}}\right)$$
where
$$dev^{*}_\ell(f) = -\hat{f}(0)\sum_{P}\frac{(\ell-1)\deg(P)}{(1+(\ell-1)|P|^{-1})(|P|^{\ell/2}-1)}$$
and the sum is over all prime polynomials with $|P|=q^{\deg(P)}$.
\end{cor}

\subsection{A Thin Family}

In this paper, we wish to consider what happens for arbitrary, non-prime $r$. As it turns out, working with the family $\Hh_r^{gen}(g)$ when $r$ is not prime can be tricky. For instance, there is not a necessarily clean cut way to determine when the set is non-empty (see Remark \ref{thingenformrem}). Therefore, we will restrict to working with a smaller set that behaves more nicely.

Since we assume that $q\equiv 1 \mod{r}$, we know by Kummer theory that all the curves of $\Hh_r$ will have affine models of the form
$$Y^r = \alpha F(X)$$
with $F$ a monic, $r^{th}$-power free polynomial and $\alpha\in \Ff_q^*/(\Ff_q^*)^r$. Since $F$ is $r^{th}$-power free, we can write $F = f_1f_2^2\cdots f_{r-1}^{r-1}$ where the $f_i$ are monic, square-free and pairwise coprime.  We will say that a curve admits a thin affine model, if it has an affine model of the form
$$Y^r = \alpha \prod_{(i,r)=1} f_i^i \footnote{We will commonly write the argument for a sum or product as just $(i,r)=1$ which is to be interpreted as the sum or product from $i=1$ to $r$ such that $(i,r)=1$.}$$
with the $f_i$ being monic, square-free and pairwise coprime and $\alpha\in\Ff_q^*/(\Ff_q^*)^r$. Moreover, for any integers $n,m$ we denote $(n,m)$ as their greatest common divisor. Therefore, we define the following set
\begin{align}
\Hh^{thin}_r(g) := \{ (C,\pi)\in\Hh^{gen}_r(g) : C \mbox{ admits a thin affine model}  \}.
\end{align}
Note that if $r=\ell$ is prime, then  in fact $\Hh^{thin}_\ell(g) = \Hh^{gen}_{\ell}(g)$. So, all our results will include the previous results in Section \ref{background}.

The condition for a curve to have a thin affine model can be seen as a restriction of the possible ramification types for finite primes. As we will see $\Hh^{thin}_r(g) \not=\emptyset$ if and only if $2g \equiv 1-s \mod{r-1}$ for some $s|r$. The value of $s$ corresponds to the different possible ramification types for the prime at infinity. Therefore, to ensure that the prime at infinity behaves like all the other primes we will restrict to the case $2g\equiv 0 \mod{r-1}$ (i.e. $s=1,r$).

\begin{thm}\label{thinthm}
Let $q\equiv 1 \mod{r}$ and $2g\equiv 0\mod{r-1}$. If $(r,n)=1$, then $\langle\Tr(\Theta_C^n) \rangle_{\Hh^{thin}_r(g)}=0$. Otherwise, for $n\not=0$,
$$\langle\Tr(\Theta_C^n) \rangle_{\Hh^{thin}_r(g)} = \frac{-1}{q^{n/2}}\sum_{\substack{s|(r,n) \\ s\not=1}}\phi(s)q^{n/s} -D_r(g,n)+ O\left(\frac{1}{gq^{\frac{n}{2}}}+q^{\frac{n}{2}-(\frac{1}{2}+\epsilon)\frac{2g}{r-1}}\right)$$
where $\phi$ is the Euler totient function and $D_r(g,n)$ is an explicit function, given in \eqref{D_r(g,n)}. In particular, $D_r(g,n)\ll nq^{-n/2}$.
\end{thm}

Comparing this to the previous results in Section \ref{background}, we see that when $r=2$, our error term is not as good. This is due to the fact that in \cite{C} and \cite{R}, they are able to use the simplicity of the square-free sieve to improve the error term. Since it is not so obvious of how to extend that technique to a general $r^{th}$-power sieve, we do not get as good of an error term in general.

However, comparing to the case of $r=\ell$, an odd prime, we see that our result is, in fact, stronger. While the error term in Theorem \ref{BCDthm} is  $\frac{n^{\ell-2}q^{n/\ell}}{gq^{n/2}}$, we obtain a (secondary) main term of the order of $\frac{q^{n/\ell}}{q^{n/2}}$ and many more, in the case of $r$ non-prime. Further, we replicate the result that the average is exactly equal to zero if $(\ell,n)=1$ as was previously shown only for $\ell=2$.

Also, note that if $\ell$ is prime, our value of $D_\ell(g,n)$ differs slightly from the previous results $D^{*}_\ell(g,n)$. In the case $\ell=2$, one can show that $D_2(g,n)\sim D^{*}_2(g,n)$ as $g$ tends to infinity. However, if $\ell$ is an odd prime, then we can show $D^{*}_{\ell}(g,n) \sim D_{\ell}(g,n)$ only if $n=o(g)$. This is still consistent with Theorem \ref{BCDthm} as $D^{*}_{\ell}(g,n)$ is only a true main term if $n\ll \log(g)$.

Therefore we see that if $r$ is odd, then the expected value of the trace will always tend to $0$, behaving like the unitary ensemble. Whereas if $r$ is even then the expected value tends to $-1$ if $n$ is even and $0$ if $n$ is odd, behaving like the symplectic ensemble. Thus, we can restate the result in terms of random matrices and then prove a one-level density result.

\begin{cor}\label{thincor}

Let $q\equiv1\mod{r}$ and $2g\equiv 0\mod{r-1}$. If $r$ is even, and $C_r\log_qg<n<(1-\epsilon)\frac{2g}{r-1}$ then
$$\langle\Tr(\Theta_C^n) \rangle_{\Hh^{thin}_r(g)} = \int_{USp(2g)} \Tr(U^n)dU + O\left( \frac{1}{g} \right)$$
and thus if $f$ is an even Schawrz test function with $\supp(\hat{f})\subset \left(\frac{-1}{r-1},\frac{1}{r-1}\right)$ then
$$\langle \mathcal{D}(L_C,f) \rangle_{\Hh^{thin}_r(g)} = \int_{USp(2g)} \mathcal{D}(U,f) dU + O\left(\frac{1}{g^{1-\epsilon}}\right).$$
Whereas if $r$ is odd and $C_r\log_qg < n< (1-\epsilon)\frac{2g}{r-1}$, then
$$\langle\Tr(\Theta_C^n) \rangle_{\Hh^{thin}_r(g)} = \int_{U(2g)} \Tr(U^n)dU + O\left( \frac{1}{g}\right)$$
and thus if $f$ is an even Schawrz test function with $\supp(\hat{f})\subset \left(\frac{-1}{r-1},\frac{1}{r-1}\right)$ then
$$\langle \mathcal{D}(L_C,f) \rangle_{\Hh^{thin}_r(g)} = \int_{U(2g)} \mathcal{D}(U,f) dU + O\left(\frac{1}{g^{1-\epsilon}}\right).$$

\end{cor}

\subsection{Refining one-level Density}

We confirm the Katz-Sarnak predictions in Corollary \ref{thincor} by showing that the one-level density is controlled either by the unitary ensemble or the symplectic ensemble. However, this seems to only use the primary main term we obtain in Theorem \ref{thinthm}; namely $-1$ if $r$ and $n$ are even and $0$ otherwise. The secondary main terms of Theorem \ref{thinthm} appear to capture the inherent $r$-cyclic nature of the family $\Hh_r^{thin}(g)$, while the Katz-Sarnak predictions only appears to capture the parity of the family $\Hh_r^{thin}(g)$ (i.e. $r$ odd or even). This subsection will be devoted to hypothesizing on a possible refinement of the Katz-Sarnak predictions.

With this in mind let $N,s$ be integers such that $\phi(s)|N$. Let $M_{(s)}(N)\subset U(N)$ be a hypothetical matrix group such that for all $n\leq \frac{N}{\phi(s)}$
$$\int_{M_{(s)}(N)} \Tr(U^n) dU = \begin{cases} -\phi(s) & s|n \\ 0 & \mbox{otherwise} \end{cases}.$$
We do not know as of now what this matrix group may be, or even that it should exist. However, it should be noted that setting $M_{(2)}(2N) = USp(2N)$ satisfies this.

Then we could rewrite the result of Theorem \ref{thinthm} as
\begin{align}\label{RefinedRMT}
\langle\Tr(\Theta_C^n) \rangle_{\Hh^{thin}_r(g)}=  \sum_{\substack{s|r \\ s\not=1}}\frac{q^{n/s}}{q^{n/2}}\int_{M_{(s)}\left(\frac{\phi(s)2g}{r-1}\right)}\Tr(U^n)dU -D_r(g,n)+ O\left(\frac{1}{gq^{\frac{n}{2}}}+q^{\frac{n}{2}-(\frac{1}{2}+\epsilon)\frac{2g}{r-1}}\right).
\end{align}

Note that the moment of the trace condition that defines $M_{(s)}(N)$ is only for $n\leq \frac{N}{\phi(s)}$. Thus, if we were able to determine explicitly the family $M_{(s)}(N)$, then we would likely be able to compute higher moments and hence form a conjecture for values of $n$ greater than $\frac{2g}{r-1}$.

Now, we wish to combine all the matrix subgroups $M_{(s)}$ into one big subgroup of the unitaries. That is, if we have $N,r$ integers such that $(r-1)|N$ let $\{e_i : i=1,\dots,N\}$ be the standard basis for $\mathbb{C}^N$. For any $s|r$, $s\not=1$, let
$$V_s = span_{\mathbb{C}} \left\{e_{\frac{r}{s}jk} : j=1,\dots,\frac{N}{r-1}, k=1,\dots,s, (k,s)=1\right\}.$$
For any $U\in U(N)$, denote
$$U_s = U|_{V_s}.$$
Then we define
$$M_r(N) := \left\{U \in U(N) : U_s \in M_{(s)}\left(\frac{\phi(s)N}{r-1}\right), s|r, s\not=1\right\}.$$

%$$M_r(N) = \left\{ U\begin{pmatrix} U_r & 0 & \cdots \\ 0 & U_s & \cdots \\ \vdots & \vdots & \ddots  \end{pmatrix}U^{-1} : U\in U(N), U_s \in M_{(s)}\left( \frac{\phi(s)N}{r-1} \right), s|r, s\not=1 \right\}.$$
%That is, the conjugates of the block diagonal matrices where the diagonals are indexed by the divisors of $r$.

We would then like to relate $\langle \mathcal{D}(L_C,f) \rangle$ to $\int_{M_r(2g)} \mathcal{D}(U,f) dU$. However, since the expected traces of the matrices in \eqref{RefinedRMT} appear with weights, we need to massage $\mathcal{D}(U,f)$ slightly.

We now define the one-level density of a $U\in M_r(N)$ weighted by $q$ as
$$\mathcal{D}_q(U,f) := \sum_{\substack{s|r\\s\not=1}} \sum_{j=1}^{\frac{N}{r-1}} \sum_{(k,s)=1} \sum_{n\in\Z} f\left( N \left(\frac{\theta_{\frac{r}{s}jk} +i\log(q)\left(\frac{1}{s}-\frac{1}{2}\right)}{2\pi}-n \right)\right)$$
where the $\theta_{\frac{r}{s}jk}$ are the eigenangles of $U$. We see that in the case $r=2$, this is the same as \eqref{D(u,f)}. Moreover, for any $s|r$, we should think of the summand as a one-level density of $U_s$.
%
%Now, for any $U\in M_r(N)$, we enumerate the eigenangles $\theta_1,\dots,\theta_N$ such that
%$$\exp\left( \mbox{diag}\left(\left(i\theta_{\frac{r}{s}j}, i\theta_{2\frac{r}{s}j}, \dots, i\theta_{\frac{N}{r-1}\frac{r}{s}j}\right)_{(j,s)=1}\right)\right) \in M_{(s)}\left(\frac{\phi(s)N}{r-1}\right).$$

\begin{thm}\label{refinedthm}

Suppose $q\equiv 1 \mod{r}$ and $2g\equiv 0\mod{r-1}$. Then if $f$ is a Schwartz test function such that $\supp(\hat{f})\subset [0,\frac{1}{r-1})$, we get that
$$\langle\mathcal{D}(L_C,f)\rangle_{\Hh_r^{thin}(g)} = \int_{M_r(2g)}\mathcal{D}_q(U,f) dU -\frac{dev_r(f)}{2g} + O\left(\frac{1}{g^2}\right)$$
where if we denote $d=\frac{2g+2r-2}{r-1}$
$$dev_r(f):=\hat{f}(0)\sum_{\substack{s|r \\ s\not=1}} \phi(s) \sum_{P} \frac{\deg(P)}{|P|^{s/2}-1} \sum_{a=1}^{\lfloor \frac{d}{\deg(P)}\rfloor} \left(\frac{-\phi(r)}{|P|}\right)^a \left(1-\frac{a\deg(P)}{d}\right)^{\phi(r)-1}$$
and the sum is over all prime polynomials with $|P|=q^{\deg(P)}$.
\end{thm}

Note that here we take $f$ such that $\supp(\hat{f})\subset[0,\frac{1}{r-1})$ and so, in particular, $f$ is not even. This is to remedy the fact that when we perform our poisson summation we will obtain summands of the form $\hat{f}(\frac{n}{2g})\frac{q^{n/s}}{q^{n/2}}$. Thus, if $n$ is negative, we would end up with a factor of $\frac{q^{n/2}}{q^{n/s}}$ in the calculation of the one-level density of the random matrices which no longer matches the one-level density of the family of $L$ functions.

Again, we see that if $\ell$ is prime, then $dev_{\ell}(f)$ differs from $dev^{*}_{\ell}(f)$ from the previous results. However, as with $D_{\ell}(g,n)$, if $\deg(P) = o(g)$, then
$$\sum_{a=1}^{\lfloor \frac{d}{\deg(P)}\rfloor} \left(\frac{-\phi(\ell)}{|P|}\right)^a \left(1-\frac{a\deg(P)}{d}\right)^{\phi(r)-1} \sim \frac{1}{1+(\ell-1)|P|^{-1}}$$
which would make the summand of $dev_{\ell}(f)$ consistent with that of $dev^{*}_{\ell}(f)$.

\section{From Curves to Polynomials}\label{polysect}

In this section we show that we can parameterize each of our families by polynomials.

\subsection{Polynomials}\label{polysubsec}

For any polynomial $F\in\Ff_q[X]$, denote by $C_{F,r}$ the curve such that
\begin{align}
\Ff_q(C_{F,r}) = \Ff_q(X)(F^{1/r})
\end{align}
where we denote $\Ff_q(C_{F,r})$ as the field of functions of the curve $C$. Since we are assuming $q\equiv 1 \mod{r}$, then $\Ff_q(X)$ contains the $r^{th}$-roots of unity and so Kummer Theory tells us that all the curves in $\Hh_r$ are of this form. However, not all polynomials will give us a curve in $\Hh_r$. For example, if $F=G^s$, for some $s|r$, then the resulting curve $C_{F,r}$ will actually be an $\frac{r}{s}$ cyclic cover.

Further not all polynomials will give different curves. For example, if $F$ and $G$ are two polynomials such that there exists an $H\in\Ff_q(X)$ and an $i$, coprime to $r$, with $F=G^iH^r$, the we would get that $C_{F,r}$ and $C_{G,r}$ would be birationally equivalent.

It is easy to see that these are the only two restraints. Therefore, we define the set
$$\F_r := \{F\in \Ff_q[X] : F \mbox{ is $r^{th}$-power free and not a power of $s$ for any $s|r$, $s\not=1$}\}$$
and we get the following lemma.
\begin{lem}\label{setlem}
There is a $\varphi(r):1$ map between the two sets
\begin{align*}
\F_r & \to  \Hh_r \\
F & \mapsto  C_{F,r}
\end{align*}
\end{lem}

\subsection{Genus Formula}\label{genformsect}

Let $F\in\F_r$. We can write
$$F = \alpha f_1f_2^2\cdots f^{r-1}_{r-1}$$
where the $f_i$ are monic, squarefree and pairwise coprime and $\alpha\in\Ff_q^*/(\Ff_q^*)^r$. Then, we apply the Riemann Hurwitz formula (Theorem 7.16 of \cite{rose}) to get that the genus of $C_{F,r}$, $g$, satsifies
\begin{align}\label{genform}
2g+2r-2 = \sum_{i=1}^{r-1} (r-(r,i))\deg(f_i) + r-(r,\deg(F))
\end{align}
where we use the convention that for every pair of integers $m$ and $n$, $(m,n)$ denotes the greatest common divisor of $m$ and $n$. Therefore, define the set
\begin{align}\label{polygen}
\F^{gen}_r(g) = \left\{ F\in\F_r : 2g+2r-2 = \sum_{i=1}^{r-1} (r-(r,i))\deg(f_i) + r-(r,\deg(F)) \right\}.
\end{align}

Then we get an immediate lemma.
\begin{lem}\label{gensetlem}
As long as $\F^{gen}_r(g)\not=\emptyset$, then there is a $\varphi(r):1$ surjective map between the two sets $\F^{gen}_r(g)$ and $\Hh^{gen}_r(g)$. In particular as long as $\F^{gen}_r(g)\not=\emptyset$, then
$$\langle \Tr(\Theta_C^n)\rangle_{\Hh^{gen}_r(g)} = \langle \Tr(\Theta_{C_{F,r}}^n) \rangle_{F^{gen}_r(g)}.$$
\end{lem}

\begin{rem}\label{thingenformrem}

We see that $\Hh^{gen}_r(g)$ is non-empty if and only if there is a solution to
$$2g = \sum_{s|r} (r-s)d_s$$
which will have a solution if and only if $g$ is sufficiently large and $\gcd_{s|r}(r-s)|2g$. While this condition can likely be dealt with, it illustrates the difficulties that can arise in working with $\Hh^{gen}_r(g)$, of which there are more that are more cumbersome to handle.

\end{rem}

\subsection{Thin Family}\label{thinfam}

Now if we consider the subset of $\F_r$
\begin{align}\label{polythin}
\F^{thin}_r = \left\{ F\in\F_r : F=\alpha \prod_{ (i,r)=1} f_i^i, \alpha\in\Ff_q^*/(\Ff_q^*)^r \right\}
\end{align}
then we see that the genus formula for curves $C_{F,r}$ with $F\in\F^{thin}_r$ simplifies to
\begin{align}\label{thingenform}
2g+2r-2 = (r-1)\sum_{(i,r)=1} \deg(f_i) + r-(r,\deg(F)).
\end{align}
And so we define the set
\begin{align}\label{polythingen}
\F^{thin}_r(g) = \left\{ F\in\F^{thin}_r : 2g+2r-2 = (r-1)\sum_{(i,r)=1} \deg(f_i) + r-(r,\deg(F)) \right\}.
\end{align}
Reducing the genus formula modulo $r-1$, we find that
\begin{align}\label{redmodr-1}
2g \equiv r-(r,\deg(F)) \equiv 1 - (r,\deg(F)) \mod{r-1}.
\end{align}
Hence we have that $\F^{thin}_r(g)\not=\emptyset$ if and only if $2g\equiv 1-s \mod{r-1}$ for some $s|r$, in which case we would necessarily have $(r,\deg(F))=s$.

This leads to an immediate lemma.
\begin{lem}\label{gensetlem}
If $2g\equiv 1-s \mod{r-1}$ for some $s|r$ then there is a $\varphi(r):1$ surjective map between the two sets $\F^{thin}_r(g)$ and $\Hh^{thin}_r(g)$. In particular
$$\langle \Tr(\Theta_C^n)\rangle_{\Hh^{thin}_r(g)} = \langle \Tr(\Theta_{C_{F,r}}^n) \rangle_{F^{thin}_r(g)}.$$
\end{lem}

\begin{rem}\label{redrem}

From now on, we will restrict to the case $2g\equiv 0 \mod{r-1}$. We see from \eqref{redmodr-1} that this is equivalent to restricting the $F\in\F^{thin}_r$ to $(\deg(F),r)\equiv 1 \mod{r-1}$ which is then equivalent to saying $(\deg(F),r) = 1$ or $r$.

\end{rem}

%\subsection{Conductor Formula}\label{condformsect}
%
%Let $F\in\F_r$. Then the conductor of $C_{F,r}$, $\N(C_{F,r})$, will be the product of all the primes that ramify in $\Ff_q[X](F^{1/r})$. Hence this will be all the primes dividing $F$ plus the prime at infinity if $r\not|\deg(F)$. That is,
%\begin{align}\label{condform}
%\deg(\N(C_{F,r})) = \deg(\rad(F)) + \begin{cases} 0 & r|\deg(F) \\ 1 & \mbox{otherwise} \end{cases}.
%\end{align}
%Therefore, if we define
%\begin{align}\label{polycond}
%\F^{cond}_r(N) = \left\{ F\in \F_r : N = \deg(\rad(F)) + \begin{cases} 0 & r|\deg(F) \\ 1 & \mbox{otherwise} \end{cases} \right\}
%\end{align}
%then we get the immediate lemma.
%
%\begin{lem}\label{condsetlem}
%There is a $\varphi(r):1$ surjective map between the two sets $\F^{cond}_r(N)$ and $\Hh^{cond}_r(N)$. In particular
%$$\langle \Tr(\Theta_C^n)\rangle_{\Hh^{cond}_r(N)} = \langle \Tr(\Theta_{C_{F,r}}^n) \rangle_{F^{cond}_r(N)}.$$
%\end{lem}

\section{Trace Formula}\label{trformsec}

\subsection{Characters}\label{charsect}
Before we develop a formula for $\Tr(\Theta_C^n)$, we will need notation for certain characters. For any $n$ let $\beta\in\Ff_{q^n}^*$ be a generator. Then for any $s|r$, we define the character on $\Ff_{q^n}$ by
\begin{align}\label{char}
\chi_{s;n}(\alpha) = \begin{cases} \xi_s^t & \alpha = \beta^t \\ 0 & \alpha=0 \end{cases}
\end{align}
where $\xi_s$ is some primitive $s^{th}$ root of unity. This is well defined since we are assuming $q\equiv 1 \mod{s}$.

Further, if $P\in\Ff_q[X]$ is any irreducible polynomial with $\alpha$ as a root, then we define the Legendre symbol of $F\in\Ff_q[X]$ mod $P$ as
\begin{align}\label{legen}
\left(\frac{F}{P}\right)_s := \chi_{s;\deg(P)}(F(\alpha)).
\end{align}

\begin{lem}\label{charlem}
Let $\alpha\in\Ff_{q^n}$ and let $P_{\alpha}$ be its minimal polynomial. Then
$$\chi_{s;n}(F(\alpha)) = \left(\frac{F}{P_{\alpha}}\right)_s^{n/\deg(P_{\alpha})}.$$
\end{lem}

\begin{proof}
If $P_{\alpha}|F$ this is trivially true as both sides are $0$. So now, assume $P_{\alpha}\nmid F$.

Suppose $\deg(P_{\alpha})=d|n$ so that $\alpha\in\Ff_{q^d}$. Let $\beta$ be a generator of $\Ff_{q^n}^*$. Therefore $\beta^{\frac{q^n-1}{q^d-1}}$ is a generator of $\Ff_{q^d}^*$. Then $F(\alpha) = \beta^{\frac{q^n-1}{q^d-1}t}$ for some $t$ and hence
$$\chi_{s;n}(F(\alpha)) = \xi_s^{\frac{q^n-1}{q^d-1}t} = \xi_s^{\frac{n}{d}t} =\left(\frac{F}{P_{\alpha}}\right)_s^{n/\deg(P_{\alpha})} $$
where the second equality is obtained from the fact the $q\equiv 1 \mod{s}$ and so $\frac{q^n-1}{q^d-1}\equiv \frac{n}{d} \mod{s}$.
\end{proof}

%We then extend this to any polynomial $G=\prod P_i^{v_i}$ by setting
%\begin{align*}
%\left(\frac{F}{G}\right)_s = \prod \left(\frac{F}{P_i}\right)_s^{v_i}.
%\end{align*}
If we denote $P_{\infty}$ as the prime at infinity, then we can define an analogous character modulo the prime at infinity:
\begin{align}\label{charP_inf}
\left(\frac{F}{P_{\infty}}\right)_s = \begin{cases} \chi_{s;1}(\mbox{leading coefficient of $F$}) & s|\deg(F) \\ 0 & \mbox{otherwise} \end{cases}.
\end{align}
In particular, we set $\deg(P_\infty)=1$ and with a similar proof as in Lemma \ref{charlem} we get that if $s|\deg(F)$,
\begin{align}\label{charP_inf2}
\chi_{s;n}(\mbox{leading coefficient of $F$}) = \left( \frac{F}{P_{\infty}}\right)_s^n = \left(\frac{F}{P_{\infty}}\right)_s^{\frac{n}{\deg(P_{\infty})}}.
\end{align}
%Finally, we just note here that since we are always assuming $q\equiv 1 \mod{2r}$, we have the following reciprocity relation
%\begin{align}\label{recipr}
%\left(\frac{F}{G}\right)_s = \left(\frac{G}{F}\right)_s
%\end{align}
%for all $s|r$.

\subsection{Trace Formula}\label{trformsubsec}

From \eqref{zetafunc},\eqref{Lfunc} and \eqref{Frob}, we have that
\begin{align}
\exp\left(\sum_{n=1}^{\infty} \#C(\Ff_{q^n}) \frac{u^n}{n} \right) = \frac{ \det(1-\sqrt{q}u\Theta_C) }{(1-u)(1-qu)}.
\end{align}
Taking logs and equating coefficients of $u$, we get the well known identity
\begin{align}\label{numptsform}
\#C(\Ff_q^n) = q^n+1-q^{n/2}\Tr(\Theta_C^n).
\end{align}
Therefore, to obtain a formula to $\Tr(\Theta_C^n)$ it is enough to obtain a formula for $\#C(\Ff_q^n)$.

This was done in \cite{M1} in the case $n=1$. However, the same method applies for arbitrary $n$. That is, for any
$$F = \prod_{i=1}^{r-1}f_i^i\in\F_r$$
then we define for any $s|r$
$$F_{(s)} = \prod_{i=1}^{s-1} \prod_{j=1}^{r/s-1} f_{sj+i}^i$$
and we can write $\#C(\Ff_{q^n})$ in terms of $F_{(s)}$.

\begin{lem}\label{numptslem}
If $C=C_{F,r}$ for some $F\in\F_r$, then
$$\#C(\Ff_q^n) =q^n+1+ \sum_{\alpha\in\Pp^1_{\Ff_{q^n}}}\sum_{s|r}\sum_{\substack{i=1 \\ (i,s)=1}}^{s-1}  \chi^i_{s;n}(F_{(s)}(\alpha))$$
where $\chi_{s;n}$ is defined in \eqref{char} and if we denote $\infty$ as the point at infinity, then
$$F_{(s)}(\infty) := \begin{cases} \mbox{leading coefficient of $F$} & \deg(F)\equiv 0 \mod{s} \\ 0 & \mbox{otherwise}  \end{cases}.$$
\end{lem}

Therefore, combining Lemmas \ref{charlem} and \ref{numptslem}, we obtain a formula for $\Tr(\Theta_C^n)$.

\begin{lem}\label{trformlem}
If $C=C_{F,r}$ for some $F\in\F_r$, then
$$-q^{n/2}\Tr(\Theta_C^n) =   \sum_{s|r} \sum_{\substack{i=1 \\ (i,s)=1}}^{s-1} \sum_{\deg(P)|n} \deg(P) \left(\frac{F_{(s)}}{P} \right)_s^{\frac{in}{d}} $$
where the sum is over prime polynomials, including the prime at infinity $P_{\infty}$.
\end{lem}

\begin{proof}

Lemma \ref{numptslem} and \eqref{numptsform} tell us that
$$-q^{n/2}\Tr(\Theta_C^n) = \sum_{s|r}\sum_{\substack{i=1 \\ (i,s)=1}}^{s-1} \sum_{\alpha\in\Pp^1_{\Ff_q^n}}  \chi^i_{s,n}(F_{(s)}(\alpha))  $$
Now if $\alpha\in\Pp^1_{\Ff_{q^n}}$ then Lemma \ref{charlem} (or \eqref{charP_inf2} for $\alpha=\infty$)  tells us that
$$\chi^i_s(F_{(s)}(\alpha)) = \left(\frac{F_{(s)}}{P_{\alpha}}\right)_s^{\frac{in}{\deg(P_{\alpha})}}$$
where $P_{\alpha}$ is the minimal polynomial of $\alpha$. Finally, as $\alpha$ runs over all $\Pp^1_{\Ff_{q^n}}$, then the primes runs over all primes with degree divisible by $n$, weighted by their degree.

%Now, for the point at infinity, we get the contribution to $-q^{n/2}\Tr(\Theta_C^n)$ will be
%\begin{align*}
%\sum_{s|r}\sum_{\substack{i=1 \\ (i,s)=1}}^{s-1}  \chi^i_s(F_{(s)}(\infty)) & = \sum_{s|(\deg(F),r)}\sum_{\substack{i=1 \\ (i,s)=1}}^{s-1} \chi^i_s(\ell.c(F)) \\
%& = \begin{cases} (\deg(F),r)-1 & \ell.c.(F) \in \Ff_q^{(\deg(F),r)} \\ -1 & \mbox{otherwise}   \end{cases}
%\end{align*}
\end{proof}

\subsection{Trace Formula for Thin Set}

We note that if $F\in\F^{thin}_r$, then we get that $F_{(s)} = FH^s$ for some $H\in\Ff_q[X]$ with the property that $\rad(H)|\rad(F)$. Hence, we get
$$\left(\frac{F_{(s)}}{P}\right)_s = \left(\frac{F}{P}\right)_s = \left(\frac{F}{P}\right)_r^{r/s}.$$
Therefore, we get a slightly easier formula for the trace for elements of the thin set.

\begin{lem}\label{trformthinlem}

If $C=C_{F,r}$ for some $F\in\F^{thin}_r$, then
$$-q^{n/2} \Tr(\Theta_C^n) = \sum_{i=1}^{r-1} \sum_{\deg(P)|n} \deg(P) \left(\frac{F}{P}\right)^{\frac{in}{\deg(P)}}_r$$
where the sum is over prime polynomials, including the prime at infinity $P_{\infty}$.
\end{lem}

\subsection{Main Term and Error Term}

Combining Lemma \ref{trformthinlem} with Section \ref{polysect}, we see that
\begin{align*}
\langle \Tr(\Theta_C^n) \rangle_{\Hh^{thin}_r(g)} = \langle \Tr(\Theta_{C_F}^n \rangle_{\F^{thin}_r(g)} =  \frac{-q^{-n/2}}{|\F^{thin}_r(g)|} \sum_{F\in\F^{thin}_r(g)} \sum_{i=1}^{r-1} \sum_{\deg(P)|n} \deg(P) \left(\frac{F}{P}\right)^{\frac{in}{\deg(P)}}_r.
\end{align*}

Hence, we will define
$$MT_r(g,n) :=  \frac{-q^{-n/2}}{|\F^{thin}_r(g)|} \sum_{\deg(P)|n} \deg(P)  \sum_{\substack{i=1 \\ r|\frac{in}{\deg(P)}}}^{r-1} \sum_{\substack{F\in\F^{thin}_r(g) \\ P\nmid F}} 1$$
and
$$ET_r(g,n):=  \frac{-q^{-n/2}}{|\F^{thin}_r(g)|} \sum_{\deg(P)|n} \deg(P) \sum_{\substack{i=1 \\ r\nmid \frac{in}{\deg(P)}}}^{r-1} \sum_{F\in\F^{thin}_r(g) } \left(\frac{F}{P}\right)_r^{in/m}. $$

\begin{rem}
We include the prime at infinity $P_{\infty}$ into our sum over primes in $MT_r(g,n)$ by saying $P_{\infty}|F$ if and only if $\deg(F)\not\equiv 0\mod{r}$.
\end{rem}

\section{Computing the Main Term}\label{mainterm}

\begin{thm}\label{maintermthm}

Assume $q\equiv1 \mod{r}$ and $2g\equiv 0\mod{r-1}$, then $MT_r(g,n)=0$ if $(r,n)=1$. Otherwise
$$MT_r(g,n) = \frac{-1}{q^{n/2}}\sum_{\substack{s|(r,n) \\ s\not=1 }} \phi(s)q^{n/s} - D_r(g,n) + O\left(\frac{n}{gq^{n/2}}\right)$$
where $D_r(g,n)$ is given by \eqref{D_r(g,n)}. In particular,
$$D_r(g,n) \ll \frac{n}{q^{n/2}}$$
\end{thm}

We see immediately that if $(r,n)=1$ then $MT_r(g,n)=0$ as the condition $r|\frac{in}{\deg(P)}$ reduces down to $r|i$ which is impossible as $i=1,\dots,r-1$. So for now on we will always assume $(r,n)\not=1$.

\subsection{Notation}\label{notation}

We will first consider only the monic polynomials of $\F^{thin}_r$ and then add back the leading coefficients later. Therefore, define
$$\widehat{\F}^{thin}_r = \{F\in\F^{thin}_r : F \mbox{ is monic} \} $$
$$\widehat{\F}^{thin}_r(d) = \{F \in \widehat{\F}^{thin}_r : \deg(\rad(F))=d\}$$
and
$$\widehat{\F}^{thin}_{r;k}(d) := \left\{ F\in \F^{thin}_r(d) : \deg(F) \equiv k \mod{r}\right\}.$$

From now on, we will always assume $2g\equiv 0\mod{r-1}$ and set $d=\frac{2g+2r-2}{r-1}$. Then by Remark \ref{redrem}, we find
\begin{align}\label{monic1}
\frac{1}{r}|\F^{thin}_r(g)| =|\widehat{\F}^{thin}_{r;0} (d)| + \sum_{(k,r)=1} |\widehat{\F}^{thin}_{r;k} (d-1)|,
\end{align}
since there are $r$ different $\alpha\in\Ff_q^*/(\Ff_q^*)^r$.

Further, for any $G\in\Ff_q[T]$, we define
$$\F^{thin}_{r}(g;G) := \left\{F\in\F^{thin}_{r}(g) : (F,G)=1 \right\} $$
and
$$\widehat{\F}^{thin}_{r;k}(d;G) := \left\{F\in\widehat{\F}^{thin}_{r;k}(d) : (F,G)=1 \right\}.   $$
Consequently, we have
\begin{align}\label{monic2}
\frac{1}{r}|\F^{thin}_r(g;G)| =|\widehat{\F}^{thin}_{r;0}(d;G)| + \sum_{(k,r)=1} |\widehat{\F}^{thin}_{r;k} (d-1;G)|.
\end{align}

The key tool of proving Theorem \ref{maintermthm} is the following proposition.
\begin{prop}\label{unramprop}
Let $P$ be a finite prime of degree $n$, and let $d=\frac{2g+2r-2}{r-1}\in\mathbb{Z}$, then
\begin{align}\label{unrampropeq}
\frac{|\F^{thin}_r(g;P)|}{|\F^{thin}_r(g)|} = \sum_{a=0}^{\lfloor \frac{d}{n}\rfloor} \left(\frac{-\phi(r)}{q^n}\right)^a \left(1-\frac{an}{d}\right)^{\phi(r)-1} + O\left(\frac{1}{gq^n}\right)
\end{align}
\end{prop}

\begin{rem}

Note that if, $n>\frac{2g+2r-2}{r-1}$, then this Proposition implies that
$$\frac{|\F^{thin}_r(g;P)|}{|\F^{thin}_r(g)|} = 1 + O\left(\frac{1}{gq^n}\right).$$
Indeed, it is fairly easy to see that this actually holds with equality as $\deg(\rad(F))\leq \frac{2g+2r-2}{r-1}$ for all $F\in\F^{thin}_r(g)$ and so is immediately coprime to any polynomial of greater degree.

Moreover, if $n= o(d)$, then we can rewrite
$$\left(1-\frac{an}{d}\right)^{\phi(r)-1} = 1 + o(1)$$
and then get
$$\frac{|\F^{thin}_r(g;P)|}{|\F^{thin}_r(g)|} \sim \frac{q^n}{q^n+\phi(r)}.$$
Further, if $n=O(1)$, then we get the error above would be $O(\frac{1}{g})$.
\end{rem}

\subsection{Generating Series}

Note that $\F^{thin}_r(g) = \F^{thin}_r(g;1)$. So it is enough to consider everything for arbitrary $G \in\Ff_q[T]$. That is, define the generating series
$$\mathcal{G}^{thin}_{r;k}(u;G) = \sum_{d=0}^{\infty} |\widehat{\F}^{thin}_{r;k}(d;G)|u^d.$$

\begin{lem}\label{genserlem}

For any $s|r$, and any primitive $s^{th}$ root of unity, $\xi_s$, define
$$H_{r;s}(u;G) = \prod_{P\nmid G} \left(1+ \frac{\phi(r)}{\phi(s)} \sum_{(i,s)=1} (\xi_s^i u)^{\deg(P)} \right).$$
Then
$$\mathcal{G}^{thin}_{r;k}(u;G) = \frac{1}{r} \sum_{s|r} \sum_{(j,s)=1} \xi_s^{-jk} H_{r;s}(u;G).$$
\end{lem}

\begin{proof}

Since every monic $F\in\widehat{\F}^{thin}_r$ can be written as
$$F = \prod_{(i,r)=1}f_i^i$$
where the $f_i$ are monic, squarefree and pairwise coprime we can write $\mathcal{G}^{thin}_{r;k}(u;G)$ as a sum over the monic $f_i$:
\begin{align*}
\mathcal{G}^{thin}_{r;k}(u;G) & = \sum_{(f_i)_{(i,r)=1}} \mu^2\left(G\prod f_i\right) \frac{1}{r} \sum_{j=0}^r \xi_r^{j(\sum i\deg(f_i)-k)} u^{\sum \deg(f_i)}\\
& = \frac{1}{r} \sum_{j=0}^r \xi_r^{-jk} \sum_{(f_i)_{(i,r)=1}} \mu^2\left(G\prod f_i\right) \prod \left( \xi_r^{ij}u \right)^{\deg(f_i)} \\
& = \frac{1}{r} \sum_{j=0}^r \xi_r^{-jk} \prod_{P\nmid G} \left(1 + \sum_{(i,r)=1} \left( \xi_r^{ij}u \right)^{\deg(P)}\right)
\end{align*}
We see that since $(i,r)=1$, the Euler product only depends on the greatest common divisor of $j$ and $r$. Therefore, if we suppose $(j,r)=\frac{r}{s}$ for some $s|r$, then we can write $j=\frac{r}{s} j'$ with $(j',s)=1$ and so
\begin{align*}
\mathcal{G}^{thin}_{r;k}(u;G) & =\frac{1}{r} \sum_{s|r} \sum_{(j',s)=1} \xi_s^{-j'k} \prod_{P\nmid G} \left(1 + \sum_{(i,r)=1} \left( \xi_s^{i}u \right)^{\deg(P)}\right)\\
& =  \frac{1}{r} \sum_{s|r} \sum_{(j,s)=1} \xi_s^{-jk} H_{r;s}(u;G)
\end{align*}

\end{proof}

\subsection{Inclusion-Exclusion for $H_{r;s}(u,P)$}

While the dependance on $G$ in $\mathcal{G}_{r;k}(u;G)$ is complicated, we can use the fact that the dependance on $H_{r;s}(u,G)$ is fairly straightforward (especially when $G=P$ is prime) to perform an inclusion exclusion.

\begin{lem}\label{incexclem}

Let $P$ be a prime of degree $n$. Then we have
$$[u^d] H_{r;s}(u;P) = \sum_{a=0}^{\lfloor \frac{d}{n}\rfloor} \left(-\frac{\phi(r)}{\phi(s)}\right)^a \sum_{\substack{i_1,\dots,i_a \\ (i_j,s)=1}} \xi_s^{n\sum i_j } [u^{d-an}] H_{r;s}(u;1) $$
where for any analytic function $K$ we denote $[u^d]K(u)$ as the $d^{th}$ coefficient of the Taylor series centered around $0$.

\end{lem}

\begin{proof}

By the definition of $H_{r;s}(u;G)$, we have that
$$H_{r;s}(u;1) = \left(1+\frac{\phi(r)}{\phi(s)} \sum_{(i,s)=1} (\xi_s^iu)^{\deg(P)}\right) H_{r;s}(u;P)$$
from which we obtain
$$[u^d]H_{r;s}(u;1) = [u^d]H_{r;s}(u;P) + \frac{\phi(r)}{\phi(s)} \sum_{(i,s)=1} \xi_s^{i\deg(P)} [u^{d-n}] H_{r;s}(u;P). $$
and so rearranging for $ [u^d]H_{r;s}(u;P)$ and recursively applying the formula, we get
\begin{align*}
[u^d]H_{r;s}(u;P) = & [u^d]H_{r;s}(u;1) - \frac{\phi(r)}{\phi(s)} \sum_{(i,s)=1} \xi_s^{i\deg(P)} [u^{d-n}] H_{r;s}(u;P)\\
= &  [u^d]H_{r;s}(u;1) - \frac{\phi(r)}{\phi(s)} \sum_{(i,s)=1} \xi_s{^i\deg(P)} [u^{d-n}]H_{r;s}(u;1)\\
& + \left(\frac{\phi(r)}{\phi(s)}\right)^2 \sum_{\substack{(i_1,s)=1 \\ (i_2,s)=1}} \xi_s^{(i_1+i_2)\deg(P)} [u^{d-2n}] H_{r;s}(u;P) \\
= & \cdots \\
= & \sum_{a=0}^{\lfloor \frac{d}{n}\rfloor} \left(-\frac{\phi(r)}{\phi(s)}\right)^a \sum_{\substack{i_1,\dots,i_a \\ (i_j,s)=1}} \xi_s^{\sum i_j \deg(P)} [u^{d-an}] H_{r;s}(u;1)
\end{align*}

\end{proof}

\subsection{Asymptotics of $[u^d]H_{r,s}(u;1)$}

We see now, it suffices to determine the size of $[u^d]H_{r,s}(u;1)$.

\begin{lem}\label{aymptlem}

There is a polynomial $P_{r,s}$ of degree at most $\frac{\phi(r)}{\phi(s)}-1$ such that
$$[u^d]H_{r,s}(u;1) = P_{r,s}(d) q^d + O\left(q^{\left(\frac{1}{2}+\epsilon\right)d}\right).$$
Moreover, $P_{r,1}$ is a polynomial of exact degree $\phi(r)-1$.
\end{lem}

\begin{proof}

We have that
\begin{align*}
H_{r;s}(u;1) & = \prod_{P} \left(1+ \frac{\phi(r)}{\phi(s)} \sum_{(i,s)=1} (\xi_s^i u)^{\deg(P)} \right) \\
& = \prod_{(i,s)=1} \prod_P \left(1 - (\xi_s^iu)^{\deg(P)} \right)^{-\frac{\phi(r)}{\phi(s)}} \widetilde{H}_{r;s}(u;1)
\end{align*}
where $\widetilde{H}_{r;s}(u;1) = \prod_P \left(1+O(u^{2\deg(P)})\right)$ and so absolutely converges for $|u|<q^{-1/2}$. Further
$$\prod_P \left(1 - (\xi_s^iu)^{\deg(P)} \right)^{-1} = \zeta_q(\xi_s^iu) = \frac{1}{1-q\xi_s^iu}$$
where $\zeta_q$ is the zeta function attached to $\Ff_q[X]$. Therefore, $H_{r;s}(u;1)$ can be meromorphically extended to the region $|u|<q^{-1/2}$ with poles of order $\frac{\phi(r)}{\phi(s)}$ when $u=(\xi_s^iq)^{-1}$. So, if $\Gamma = \{u : |u|=q^{-1/2-\epsilon}\}$, then by Cauchy residue formula, we get
$$\frac{1}{2\pi i}\oint_{\Gamma} \frac{H_{r;s}(u;1)}{u^{d+1}}du = \Res_{u=0} \left(\frac{H_{r;s}(u;1)}{u^{d+1}}\right) + \sum_{(i,s)=1} \Res_{u=(\xi_s^iq)^{-1}} \left( \frac{H_{r;s}(u;1)}{u^{d+1}} \right)$$

By definition, we have
$$\Res_{u=0} \left(\frac{H_{r;s}(u;1)}{u^{d+1}}\right) = [u^d]H_{r;s}(u;1).$$
Further, since $\widetilde{H}_{r;s}(u,1)$ absolutely converges on $\Gamma$, it will be absolutely bounded as well. Therefore
$$\frac{1}{2\pi i}\oint_{\Gamma} \frac{H_{r;s}(u;1)}{u^{d+1}}du  \ll q^{\left(\frac{1}{2} +\epsilon\right)d}.$$

So it remains to determine the residues at $(\xi_s^iq)^{-1}$. By the formula of higher order poles, if we denote $M:=\frac{\phi(r)}{\phi(s)}$, then
\begin{align*}
\Res_{u=(\xi_s^iq)^{-1}} \left( \frac{H_{r;s}(u;1)}{u^{d+1}} \right) & = \frac{1}{(M-1)!}\lim_{u\to (\xi_s^iq)^{-1}} \frac{d^{M-1}}{du^{M-1}} \left((u-(\xi_s^iq)^{-1})^M \frac{H_{r;s}(u;1)}{u^{d+1}} \right) \\
& = \frac{1}{(M-1)!}\lim_{u\to (\xi_s^iq)^{-1}} \frac{d^{M-1}}{du^{M-1}} \frac{K(u)}{u^{d+1}}\\
& = \frac{1}{(M-1)!} \sum_{m=0}^{M-1} \binom{M-1}{m} K^{(M-m)}\left((\xi_s^iq)^{-1}\right) (d+1)\cdots (d+m) (\zeta_s^iq)^{d+m+1}\\
& = Q_i(d)q^d
\end{align*}
where
$$K(u) = (-(\xi_s^iq)^{-1})^M \prod_{\substack{(j,s)=1 \\ j\not=i}} \frac{1}{1-q\xi_s^ju} \widetilde{H}_{r;s}(u;1)$$
and $Q_i$ is a polynomial of exact degree $\frac{\phi(r)}{\phi(s)}-1$.

Setting $P_{r,s} = \sum_{(i,s)=1} Q_i$ finishes the proof.

\end{proof}

\begin{cor}\label{aymptcor}
Suppose $q\equiv 1 \mod{r}$ and $2g\equiv 0 \mod{r-1}$, then there is a polynomial $P_r$ of degree $\phi(r)-1$ such that
$$|\F^{thin}_r(g)| = P_r(g)q^{\frac{2g+2r-2}{r-1}} + O\left(q^{\left(\frac{1}{2}+\epsilon\right)\frac{2g}{r-1}}\right).$$
In particular, the leading coefficient of $P_r$ will be
$$C_r := C_{r,1}\left(\frac{2}{r-1}\right)^{\phi(r)-1} \frac{q+\phi(r)}{q} \not=0$$
where $C_{r,1}$ is the leading coefficient of $P_{r,1}$ as in Lemma \ref{aymptlem}.
\end{cor}

\begin{proof}
Recall that if $d=\frac{2g+2r-2}{r-1}$ then
\begin{align}\label{a=0form}
\frac{1}{r}|\F^{thin}_r(g)| = &  |\widehat{\F}^{thin}_{r;0}(d)| + \sum_{(k,r)=1} | \widehat{\F}^{thin}_{r;k}(d-1)| \nonumber\\
= & \left[u^{d}\right] \mathcal{G}^{thin}_{r;0}(u;1) + \sum_{(k,r)=1} \left[u^{d-1}\right] \mathcal{G}^{thin}_{r;k}(u;1) \nonumber \\
= & \frac{1}{r} \sum_{s|r} \sum_{(j,s)=1} \left[u^{d}\right] H_{r;s}(u;1) + \frac{1}{r} \sum_{s|r} \sum_{\substack{(k,r)=1 \\ (j,s)=1}} \xi_s^{-jk}\left[u^{d-1}\right] H_{r;s}(u;1)\nonumber \\
= & \frac{1}{r} \sum_{s|r}\phi(s) [u^{d}] H_{r;s}(u;1) + \frac{\phi(r)}{r} \sum_{s|r}\mu(s)  [u^{d-1}] H_{r;s}(u;1)\\
= & \frac{1}{r} \sum_{s|r} \phi(s) P_{r,s}(d)q^{d} + \frac{\phi(r)}{r} \sum_{s|r} \mu(s) P_{r,s}(d-1) q^{d-1}+ O\left( q^{(\frac{1}{2}+\epsilon)d} \right)\nonumber
\end{align}
with $P_{r,s}$ as in Lemma \ref{aymptlem}. Hence, setting
\begin{align}\label{P_r(g)}
P_r(g) = \sum_{s|r} \phi(s) P_{r,s}(d) + \frac{\phi(r)}{q} \sum_{s|r} \mu(s) P_{r,s}(d-1)
\end{align}
we get that $P_r(g)$ will be a polynomial of degree at most $\phi(r)-1$. Further, we see that the $\phi(r)-1$ coefficient is obtained only when $s=1$ and so will be
$$C_r = C_{r,1}\left(\frac{2}{r-1}\right)^{\phi(r)-1} \frac{q+\phi(r)}{q} \not=0  $$
where $C_{r,1}$ is the leading coefficient of $P_{r,1}$.
\end{proof}

\subsection{Proof of Proposition \ref{unramprop}}

\begin{proof}[proof of Proposition \ref{unramprop}]

Combining Lemmas \ref{genserlem} and \ref{incexclem} and recalling that $d=\frac{2g+2r-2}{r-1}$,
\begin{align*}
\frac{1}{r}|\F^{thin}_r(g;P)| = & \frac{1}{r} \sum_{s|r} \phi(s) [u^d] H_{r,s}(u;P) + \frac{\phi(r)}{r} \sum_{s|r}\mu(s)[u^{d-1}] H_{r,s}(u;P) \\
= & \frac{1}{r} \sum_{s|r}\phi(s) \sum_{a=0}^{\lfloor \frac{d}{n}\rfloor} \left(-\frac{\phi(r)}{\phi(s)}\right)^a \sum_{\substack{i_1,\dots,i_a \\ (i_j,s)=1}} \xi_s^{n\sum i_j } [u^{d-an}] H_{r;s}(u;1) \\
& + \frac{\phi(r)}{r} \sum_{s|r}\mu(s) \sum_{a=0}^{\lfloor \frac{d-1}{n}\rfloor} \left(-\frac{\phi(r)}{\phi(s)}\right)^a \sum_{\substack{i_1,\dots,i_a \\ (i_j,s)=1}} \xi_s^{n\sum i_j } [u^{d-1-an}] H_{r;s}(u;1).
\end{align*}
We will split this sum into three cases:
\begin{enumerate}
\item $a=0$
\item $s\not=1$, $a\not=0$
\item $s=1$, $a\not=0$
\end{enumerate}

Case 1: $a=0$

In this case, we get a contribution of
$$\frac{1}{r} \sum_{s|r}\phi(s) [u^{d}] H_{r;s}(u;1) + \frac{\phi(r)}{r} \sum_{s|r}\mu(s)  [u^{d-1}] H_{r;s}(u;1)$$
which we see by \eqref{a=0form} is exactly $\frac{1}{r}|\F^{thin}_r(g)|$.
%
%Case 2: $s=1$, $a > \frac{d}{2n}$
%
%We have by Lemma \ref{aymptlem} that
%$$[u^{d-an}]H_{r;1}(u) \ll (d-an)^{\phi(r)-1} q^{d-an} \ll d^{\phi(r)-1}q^{d/2}$$
%and so the total contribution from these terms will be $O\left(d^{\phi(r)-1}r^{d/n}q^{d/2}\right)$.

Case 2: $s\not=1$, $a\not=0$

By Lemma \ref{aymptlem} we have that
$$[u^{d-an}]H_{r;s}(u) \ll (d-an)^{\frac{\phi(r)}{\phi(s)}-1} q^{d-an} \ll d^{\phi(r)-2}q^{d-an}$$
and so since we are assume $q\equiv1 \mod{r}$, we get $r\leq q$ and the total contribution from these terms will be $O_r\left(d^{\phi(r)-2}q^{d-n}\right)$.

Case 3: $s=1$, $a\not=0$

Finally, by Lemma \ref{aymptlem}, the rest is
\begin{align*}
&\frac{1}{r} \sum_{a=1}^{\lfloor \frac{d}{n}\rfloor} \left(-\phi(r)\right)^a\left( [u^{d-an}]H_{r;1}(u;1)+\phi(r) [u^{d-1-an}] H_{r;1}(u;1)\right)\\
= & \frac{q^d}{r}C_{r,1} \frac{q+\phi(r)}{q} \sum_{a=1}^{\lfloor \frac{d}{n}\rfloor} \left( \frac{-\phi(r)}{q^n} \right)^a  (d-an)^{\phi(r)-1} + O\left( d^{\phi(r)-2}q^{d-n}\right).
\end{align*}

Now, using the Corollary \ref{aymptcor}, we have that
$$|\F_r^{thin}(g)| = C_{r,1} \frac{q+\phi(r)}{q} d^{\phi(r)-1}q^d + O\left( d^{\phi(r)-2}q^d\right)$$
and so
\begin{align*}
\frac{|\F_r^{thin}(g;P)|}{|\F_r^{thin}(g)|} = 1 + \sum_{a=1}^{\lfloor \frac{d}{n}\rfloor} \left(\frac{-\phi(r)}{q^n}\right)^a \left(1-\frac{an}{d}\right)^{\phi(r)-1} + O\left(\frac{1}{gq^n}\right)
\end{align*}

\end{proof}

\subsection{Prime at Infinity}

Recall that if $P_{\infty}$ is the prime at infinity then we say $P_\infty|F$ if and only if $\deg(F)\not\equiv0\mod{r}$. Hence, we get that
\begin{align}\label{infprimram}
\widehat{\F}_r^{thin}(d,P_\infty) &  := \{ F\in\widehat{\F}_r^{thin}(d) : (F,P_{\infty})=1 \} \\
& = \{ F\in\widehat{\F}_r^{thin}(d) : P_\infty \nmid F \} \nonumber \\
& = \widehat{\F}^{thin}_{r;0}(d) \nonumber
\end{align}

Thus, we can extend Propositon \ref{unramprop} to include the prime at infinity as well.

\begin{lem}\label{unrinflem}
If $q\equiv 1\mod{r}$ and $2g\equiv 0 \mod{r-1}$, then
\begin{align}\label{unrinfeq}
\frac{|\F^{thin}_{r}(g,P_{\infty})|}{|\F^{thin}_r(g)|} = \frac{q}{q+\phi(r)} + O\left( \frac{1}{g} \right).
\end{align}

\end{lem}

\begin{proof}

By \eqref{infprimram}, we see that since there are $r$ different possible leading coefficients, $|\F^{thin}_r(g,P_{\infty})| = r| \widehat{\F}^{thin}_{r;0}(d)|$ where $d=\frac{2g+2r-2}{r-1}$. Therefore,
\begin{align*}
\frac{|\F^{thin}_{r}(g,P_{\infty})|}{|\F^{thin}_r(g)|} & = \frac{r|\widehat{\F}^{thin}_{r;0}(d)|}{|\F^{thin}_r(g)|}  = \frac{ \sum_{s|r}\phi(s) P_{r,s}(d)}{ P_r(g) } + O\left( q^{-(\frac{1}{2}-\epsilon)\frac{2g}{r-1}} \right) \\
& = \frac{C_{r,1}}{C_r} + O\left(\frac{1}{g}\right) \\
& = \frac{q}{q+\phi(r)} + O\left(\frac{1}{g}\right)
\end{align*}
where $P_{r,s}$ and $P_r$ are as in Lemma \ref{aymptlem} and Corollary \ref{aymptcor}, respectively.

\end{proof}

For homogeneity, we can rewrite the right hand side of \eqref{unrinfeq} to look more like \eqref{unrampropeq}.

\begin{cor}\label{unrinfcor}
Let $q\equiv 1 \mod{r}$ and $2g\equiv 0\mod{r-1}$. Then
\begin{align}
\frac{|\F^{thin}_r(g;P_{\infty})|}{|\F^{thin}_r(g)|} = \sum_{a=0}^{d } \left(\frac{-\phi(r)}{q}\right)^a \left(1-\frac{a}{d}\right)^{\phi(r)-1} + O_r\left(\frac{1}{gq}\right).
\end{align}
where $d=\frac{2g+2r-2}{r-1}$. Comparing this to Proposition \ref{unramprop} we see it is equivalent to setting $n=\deg(P_{\infty})=1$.
\end{cor}

\begin{proof}
This follows from the fact that
$$\sum_{a=0}^{ d } \left(\frac{-\phi(r)}{q}\right)^a \left(1-\frac{a}{d}\right)^{\phi(r)-1} = \frac{q}{q+\phi(r)} + O\left(\frac{1}{g}\right)$$
and that $O\left(\frac{1}{g}\right) = O\left(\frac{1}{gq}\right)$ since $q$ is constant.

\end{proof}

\subsection{Proof of Theorem \ref{maintermthm}}

\begin{proof}[proof of Theorem \ref{maintermthm}]

By Proposition \ref{unramprop} for finite primes and Corollary \ref{unrinfcor} for the infinite prime, we get
\begin{align*}
MT_r(g,n)  = &  \frac{-q^{-n/2}}{|\F^{thin}_r(g)|} \sum_{\deg(P)|n} \deg(P)  \sum_{\substack{i=1 \\ r|\frac{in}{\deg(P)}}}^{r-1} \sum_{\substack{F\in\F^{thin}_r(g) \\ P\nmid F}} 1 \\
 = & -q^{-n/2} \sum_{\deg(P)|n}\deg(P) \left(\left(r,\frac{n}{\deg(P)}\right)-1\right) \sum_{a=0}^{\lfloor \frac{d}{\deg(P)}\rfloor} \left(\frac{-\phi(r)}{q^{\deg(P)}}\right)^a \left(1-\frac{a\deg(P)}{d}\right)^{\phi(r)-1} \\
 & +O\left( \frac{1}{gq^{n/2}} \sum_{\deg(P)|n}\frac{\deg(P)}{q^{\deg(P)}} \right)
\end{align*}

Now, setting
\begin{align}\label{D_r(g,n)}
D_r(g,n) := q^{-n/2} \sum_{\deg(P)|n}\deg(P) \left(\left(r,\frac{n}{\deg(P)}\right)-1\right) \sum_{a=1}^{\lfloor \frac{d}{\deg(P)}\rfloor} \left(\frac{-\phi(r)}{q^{\deg(P)}}\right)^a \left(1-\frac{a\deg(P)}{d}\right)^{\phi(r)-1}
\end{align}
we find
$$D_r(g,n) \ll q^{-n/2} \sum_{\deg(P)|n} \frac{\deg(P)}{q^{\deg(P)}} \ll q^{-n/2} \sum_{m|n} \frac{m}{q^m} \pi_q(m) \ll q^{-n/2} \sum_{m|n} 1 \ll nq^{-n/2}$$
where $\pi_q(m)$ is the number of prime polynomials of degree $m$ and hence $\pi_q(m) = \frac{q^m}{m} + O(q^{m/2})$.

Likewise, we then get that the error term will be
$$\frac{1}{gq^{n/2}} \sum_{\deg(P)|n}\frac{\deg(P)}{q^{\deg(P)}} \ll \frac{n}{gq^{n/2}}$$

It remains to consider the contribution from the $a=0$ term:
\begin{align}\label{a=0}
-q^{-n/2} \sum_{\deg(P)|n}\deg(P) \left(\left(r,\frac{n}{\deg(P)}\right)-1\right).
\end{align}
We know that
$$\sum_{\deg(P)|n}\deg(P)=q^n.$$
Therefore, we need only consider
\begin{align}\label{MTsum1}
\sum_{\deg(P)|n}\deg(P) \left(r,\frac{n}{\deg(P)}\right) & = \sum_{m|n} m\pi_q(m) \left(r,\frac{n}{m}\right) = \sum_{s|r} s \sum_{\substack{m|n \\ \left(r,\frac{n}{m}\right) =s }} m\pi_q(m)
\end{align}
Now, $\left(r,\frac{n}{m}\right)=s$ if and only if  $m|\frac{n}{s}$ and $\left(\frac{n}{ms},\frac{r}{s}\right)=1$. Hence, we can rewrite \eqref{MTsum1} as
\begin{align*}
\sum_{s|r}s \sum_{m|\frac{n}{s}} m\pi_q(m) \sum_{s'|\left(\frac{n}{ms},\frac{r}{s}\right)}\mu(s') & = \sum_{ss'|r}s \sum_{m|\frac{n}{ss'}}\mu(s') m\pi_q(m)\\
& = \sum_{s''|r}s'' \sum_{m|\frac{n}{s''}} m\pi_q(m) \sum_{s'|s''}\frac{\mu(s')}{s'} \\
& = \sum_{s''|r} \phi(s'') \sum_{m|\frac{n}{s''}} m\pi_q(m) \\
& = \sum_{s''|(r,n)} \phi(s'') q^{n/s''}
\end{align*}
where the last equality follows from the fact that if $s''\nmid n$, then the inner sum is empty.

Therefore, we can rewrite \eqref{a=0} to get that the $a=0$ term contributes a total of
$$-q^{-n/2} \sum_{\deg(P)|n}\deg(P) \left(\left(r,\frac{n}{\deg(P)}\right)-1\right) = \frac{-1}{q^{n/2}}\sum_{\substack{s|(r,n) \\ s\not=1}} \phi(s) q^{n/s} $$

\end{proof}

\section{Computing the Error Term}

Recall
\begin{align*}
ET_r(g,n) &=  \frac{-q^{-n/2}}{|\F^{thin}_r(g)|} \sum_{\deg(P)|n} \deg(P) \sum_{\substack{i=1 \\ r\nmid \frac{in}{\deg(P)}}}^{r-1} \sum_{F\in\F^{thin}_r(g) } \left(\frac{F}{P}\right)_r^{\frac{in}{\deg(P)}}\\
\end{align*}

In this section we will prove the following theorem.

\begin{thm}\label{ETthm}

Let $q\equiv 1 \mod{r}$ and $2g\equiv 0 \mod{r-1}$. If $(r,n)=1$ we get $ET_r(g,n)=0$. Otherwise
$$ET_r(g,n) \ll \frac{((r,n)-1)q^{n/2}}{q^{(\frac{1}{2}-\epsilon)\frac{2g}{r-1}}}$$

\end{thm}

\subsection{The Prime at Infinity}

Recall that if $P_{\infty}$ is the prime at infinity and $F\in\widehat{\F}_r^{thin}$ then
$$\left( \frac{\alpha F}{P_{\infty}}\right)_r = \begin{cases} \chi_{r;1}(\alpha) & F\in\widehat{\F}^{thin}_{r;0} \\ 0 & \mbox{otherwise} \end{cases}.$$
Thus if $d=\frac{2g+2r-2}{r-1}$, we get the contribution to $ET_r(g,n)$ from the prime at infinity will be
$$-\frac{|\widehat{\F}_{r;0}^{thin}(d)|}{q^{n/2}|\F_r^{thin}(g)} \sum_{\substack{i=1 \\ r\nmid in}}^{r-1}\sum_{\alpha\in\Ff_q^*/(\Ff_q^*)^r} \chi_{r;1}^{in}(\alpha).$$
Now, if we let $\beta$ be a generator of $\Ff_q^*$, then we find
$$\sum_{\alpha\in\Ff_q^*/(\Ff_q^*)^r} \chi_{r;1}^{in}(\alpha) = \sum_{j=0}^{r-1} \chi_{r;1}^{in}(\beta^j) = \sum_{j=0}^{r-1} \xi_r^{ijn} = 0 $$
since we are assuming $r\nmid in$ and $\xi_r$ is a primitive $r^{th}$ root of unity. Hence, the prime at infinity doesn't contribute to $ET_r(g,n)$ and so we may only consider the finite primes.

\subsection{Reducing to Monic Polynomials}

We may rewrite
\begin{align*}
ET_r(g,n) &=  \frac{-q^{-n/2}}{|\F^{thin}_r(g)|} \sum_{\deg(P)|n} \deg(P) \sum_{\substack{i=1 \\ r\nmid \frac{in}{\deg(P)}}}^{r-1} \sum_{\substack{F\in\F^{thin}_r(g) \\  monic}}  \left(\frac{ F}{P}\right)_r^{\frac{in}{\deg(P)}} \sum_{\alpha\in \Ff_q^*/(\Ff_q^*)^r} \left(\frac{\alpha}{P}\right)_r^{\frac{in}{\deg(P)}}
\end{align*}
where here the sum of primes is just over finite primes.

Now, if $\beta$ is a generator for $\Ff_{q^{\deg(P)}}^*$, then $\beta^{\frac{q^{\deg(P)}-1}{q-1}}$ will be a generator for $\Ff_q^*$ and so we may write
\begin{align*}
\sum_{\alpha\in \Ff_q^*/(\Ff_q^*)^r}\left(\frac{\alpha}{P}\right)_r^{\frac{in}{\deg(P)}}& = \sum_{j=0}^{r-1} \left(\frac{\beta^{j\frac{q^{\deg(P)}-1}{q-1}}}{P}\right)_r^{\frac{in} {\deg(P)}} = \sum_{j=0}^{r-1} \xi_r^{\frac{ijn}{\deg(P)}\frac{q^{\deg(P)}-1}{q-1}} = \sum_{j=0}^{r-1} \xi_r^{ijn}\\
& = \begin{cases} r & r|in \\ 0 & \mbox{otherwise}  \end{cases}.
\end{align*}
Thus, we must have that $i = \frac{r}{(r,n)}j$ for some $j=1,\dots,(r,n)-1$ such that $(r,n)\nmid \frac{jn}{\deg(P)}$. Hence, we may rewrite
\begin{align*}
ET_r(g,n) &=  \frac{-rq^{-n/2}}{|\F^{thin}_r(g)|} \sum_{\deg(P)|n} \deg(P) \sum_{\substack{j=1 \\(r,n)\nmid \frac{jn}{\deg(P)}}}^{(r,n)-1} \sum_{\substack{F\in\F^{thin}_r(g) \\  monic}}  \left(\frac{ F}{P}\right)_{(r,n)}^{\frac{jn}{\deg(P)}}
\end{align*}
where, again, the prime sum is only over finite primes. In particular, if $(r,n)=1$, then we see there are no such $j$ and so $ET_r(g,n)=0$ which proves the first statement of Theorem \ref{ETthm}

\subsection{Proof of Theorem \ref{ETthm}}

With the same notation as in Section \ref{notation}, we define
\begin{align}
S_{j;k}(d;P) := \sum_{F\in\widehat{\F}^{thin}_{r;k}(d)} \left(\frac{F}{P}\right)_{(r,n)}^{\frac{jn}{\deg(P)}}.
\end{align}
Then if we denote $d=\frac{2g+2r-2}{r-1}$, we get
$$ET_r(g,n) =  \frac{-rq^{-n/2}}{|\F^{thin}_r(g)|} \sum_{\deg(P)|n} \deg(P) \sum_{\substack{j=1 \\(r,n)\nmid \frac{jn}{\deg(P)}}}^{(r,n)-1} \left(S_{j;0}(d;P) + \sum_{(k,r)=1} S_{j;k}(d-1;P) \right) $$
So it remains to determine the growth of $S_{j;k}(d;P)$.

Define the generating series
$$\mathcal{K}_{j;k}(u;P) := \sum_{d=0}^{\infty} S_{j;k}(d;P)u^d .$$

\begin{lem}
The function $\mathcal{K}_{j;k}(u;P)$ is analytic in the region $|u|< q^{-1}$ and can be analytically continued to the region $|u|<q^{-1/2}$.
\end{lem}

\begin{proof}

Since $\widehat{\F}^{thin}_{r;k}(d)$ can be given by tuples of square free polynomials we can write
\begin{align*}
\mathcal{K}_{j;k}(u;P) & = \sum_{(f_i)_{(i,r)=1}} \mu^2 \left(\prod f_i\right) \frac{1}{r} \sum_{t=0}^{r-1} \xi_r^{t\left(\sum i\deg(f_i)-k\right)} \left( \frac{\prod f_i^i}{P} \right)_{(r,n)}^{\frac{jn}{\deg(P)}} u^{\sum \deg(f_i)}  \\
& = \frac{1}{r} \sum_{t=0}^{r-1} \xi_r^{-tk} \sum_{(f_i)_{(i,r)=1}} \mu^2 \left(\prod f_i\right) \prod_{(i,r)=1} \left(\left(\frac{f_i}{P}\right)_{(r,n)}^{\frac{ijn}{\deg(P)}} \left(\xi_r^{ti}u\right)^{\deg(f_i)}\right)\\
& = \frac{1}{r} \sum_{t=0}^{r-1} \xi_r^{-tk} \prod_Q \left( 1 + \sum_{(i,r)=1} \left(\frac{Q}{P}\right)_{(r,n)}^{\frac{ijn}{\deg(P)}} \left(\xi_r^{ti}u\right)^{\deg(Q)}  \right)\\
& = \frac{1}{r} \sum_{t=0}^{r-1} \xi_r^{-tk} \prod_{(i,r)=1} \prod_Q \left(1 - \left(\frac{Q}{P}\right)_{(r,n)}^{\frac{ijn}{\deg(P)}} \left(\xi_r^{ti}u\right)^{\deg(Q)}\right)^{-1} \widetilde{\mathcal{K}}_{j;k}(u;P)
\end{align*}
where the product is over all monic prime polynomials and $\widetilde{\mathcal{K}}_{j;k}(u;P) = \prod_Q \left(1+O(u^{2\deg(Q)})\right)$ and so absolutely converges for $|u|<q^{-1/2}$. Further, since the first infinite product is of the form $\prod_Q \left(1+O(u^{\deg(Q)})\right)$, we get that $\mathcal{K}_{j;k}(u;P)$ absolutely converges for $|u|<q^{-1}$.

Now, we see that
$$\prod_Q \left(1 - \left(\frac{Q}{P}\right)_{(r,n)}^{\frac{ijn}{\deg(P)}} \left(\xi_r^{ti}u\right)^{\deg(Q)}\right)^{-1} = L\left(\xi_r^{ti}u, \left(\frac{\cdot}{P}\right)_{(r,n)}^{\frac{ijn}{\deg(P)}}\right)$$
is just the $L$-function of the Dirichlet character $\left(\frac{\cdot}{P}\right)_{(r,n)}^{\frac{ijn}{\deg(P)}}$. Further, since we always have $(r,n)\nmid \frac{jn}{\deg(P)}$ and $(i,r)=1$, this will be a non-trivial Dirichlet character. Therefore, this infinite product can be extended to an entire function and hence $\mathcal{K}_{j;k}(u;P)$ can be analytically extended to the region $|u|<q^{-1/2}$.
\end{proof}

\begin{cor}
$$S_{j;k}(d;P) \ll q^{(\frac{1}{2}+\epsilon)d}$$
\end{cor}\label{ETcor}

\begin{proof}

Let $\Gamma = \{u : |u|= q^{-1/2-\epsilon}\}$. Then $\mathcal{K}_{j;k}(u;P)$ can be analytically extended to the interior of $\Gamma$ and so
$$S_{j;k}(d;P) = \frac{1}{2\pi i} \oint_{\Gamma} \frac{\mathcal{K}_{j;k}(u;P)}{u^{d+1}}du \ll q^{(\frac{1}{2}+\epsilon)d}$$

\end{proof}

\begin{proof}[proof of Theorem \ref{ETthm}]

By Corollary \ref{ETcor}, we get that
\begin{align*}
ET_r(g,n) & =  \frac{-rq^{-n/2}}{|\F^{thin}_r(g)|} \sum_{\deg(P)|n} \deg(P) \sum_{\substack{j=1 \\(r,n)\nmid \frac{jn}{\deg(P)}}}^{(r,n)-1} \left(S_{j;0}(d;P) + \sum_{(k,r)=1} S_{j;k}(d-1;P) \right) \\
& \ll \frac{q^{-n/2}}{|\F^{thin}_r(g)|} \sum_{\deg(P)|n} \deg(P) \sum_{\substack{j=1 \\(r,n)\nmid \frac{jn}{\deg(P)}}}^{(r,n)-1} q^{(\frac{1}{2}+\epsilon)d}\\
& \ll \frac{(r,n)-1}{q^{\frac{n}{2} + (\frac{1}{2}-\epsilon)d}} \sum_{\deg(P)|n}\deg(P) = \frac{((r,n)-1)q^{n/2}}{q^{(\frac{1}{2}-\epsilon)d}}
\end{align*}

\end{proof}

\section{Proof of Main Results}

\subsection{Proof of Theorem \ref{thinthm} and Corollary \ref{thincor}}

\begin{proof}[Proof of Theorem \ref{thinthm}]

By the definition of $MT_r(g,n)$ and $ET_r(g,n)$, we get
\begin{align*}
\langle \Tr(\Theta_C^n) \rangle_{\Hh^{thin}_r(g)} & = MT_r(g,n) + ET_r(g,n) \\
 & = \frac{-1}{q^{n/2}}\sum_{\substack{s|(r,n) \\ s\not=1}}\phi(s)q^{n/s} -D_r(g,n)+ O\left(\frac{1}{gq^{\frac{n}{2}}}+q^{\frac{n}{2}-(\frac{1}{2}+\epsilon)\frac{2g}{r-1}}\right)
\end{align*}
where the second line follows from Theorems \ref{maintermthm} and \ref{ETthm}. Moreover, if $(r,n)=1$, then both $MT_r(g,n)=ET_r(g,n)=0$ and so $\langle \Tr(\Theta_C^n) \rangle_{\Hh^{thin}_r(g)}=0$ as well.

\end{proof}

\begin{proof}[Proof of Corollary \ref{thincor}]

Applying Poisson summation one can show that for any unitary matrix $U$ and any even Schwartz test function, we have
$$\mathcal{D}(U,f) = \frac{1}{2g}\hat{f}(0) + \frac{1}{g}\sum_{n=1}^{\infty} \hat{f}\left(\frac{n}{2g}\right) \Tr(U^n). $$

Now, since $D_r(g,n) \ll nq^{-n/2}$ then as long as $C_r\log_q(g) \leq n \leq (1-\epsilon)\frac{2g}{r-1}$, we can use Theorem \ref{thinthm} to write
\begin{align*}
\langle \Tr(\Theta_C^n) \rangle_{\Hh^{thin}_r(g)} & = \begin{cases} -1 & \mbox{ $r$ and $n$ even} \\ 0 & \mbox{ otherwise} \end{cases} + O\left(\frac{1}{g}\right)\\
& = \begin{cases} \int_{USp(2g)} Tr(U^n) dU & r \mbox{ even} \\ \int_{U(2g)} Tr(U^n) dU & r \mbox{odd} \end{cases} + O\left(\frac{1}{g}\right)
\end{align*}
Suppose $r$ is even. Then, if we assume $\supp(\hat{f}) \subset \left(-\alpha,\alpha\right)$, for some $\alpha<\frac{1}{r-1}$  we get
\begin{align*}
\langle \mathcal{D}(L_C,f) \rangle_{\Hh^{thin}_r(g)} & = \langle \mathcal{D}(\Theta_C,f) \rangle_{\Hh^{thin}_r(g)}\\
& = \hat{f}(0) + \frac{1}{g}\sum_{n=1}^{\infty} \hat{f}\left(\frac{n}{2g}\right) \langle \Tr(\Theta_C^n) \rangle_{\Hh^{thin}_r(g)}\\
& = \hat{f}(0) + \frac{1}{g}\sum_{n=1}^{2\alpha g} \hat{f}\left(\frac{n}{2g}\right) \int_{USp(2g)} Tr(U^n) dU + O\left(\frac{1}{g} \sum_{n=1}^{C_r\log_q(g)} \hat{f}\left(\frac{n}{2g}\right) \frac{n}{q^{n/2}} + \frac{1}{g^2}\right) \\
& = \int_{USp(2g)} \left(\hat{f}(0) + \frac{1}{g}\sum_{n=1}^{\infty} \hat{f}\left(\frac{n}{2g}\right)  Tr(U^n) \right)dU + O\left(\frac{1}{g^{1-\epsilon}}\right) \\
& = \int_{USp(2g)} \mathcal{D}(U,f) dU + O\left(\frac{1}{g^{1-\epsilon}}\right).
\end{align*}

The same argument works for $r$ odd with $USp(2g)$ replaced with $U(2g)$.

\end{proof}

\subsection{Proof of Theorem \ref{refinedthm}}

\begin{proof}[Proof of Theorem \ref{refinedthm}]
For any $\theta,s$ let
$$h(x) := f\left( 2g \left(\frac{\theta +i\log(q)\left(\frac{1}{s}-\frac{1}{2}\right)}{2\pi}-x \right)\right).$$
Then
$$\hat{h}(\xi) = \frac{q^{\frac{\xi}{s}}}{q^{\frac{\xi}{2}}} \frac{e^{i \theta\xi}}{2g}\hat{f}\left(\frac{\xi}{2g}\right).$$
Now, applying Poisson summation, and using the restriction of the support of $\hat{f}$, we get
\begin{align*}
\sum_{n\in\mathbb{Z}}f \left( 2g \left(\frac{\theta +i\log(q)\left(\frac{1}{s}-\frac{1}{2}\right)}{2\pi}-n \right)\right) & = \sum_{n\in\mathbb{Z}} h(n) = \sum_{n\in\mathbb{Z}} \hat{h}(n)  \\
& = \sum_{n\in\mathbb{Z}} \frac{q^{\frac{n}{s}}}{q^{\frac{n}{2}}} \frac{e^{i \theta n}}{2g}\hat{f}\left(\frac{n}{2g}\right) \\
& = \sum_{n=0}^{\frac{2g}{r-1}} \frac{q^{\frac{n}{s}}}{q^{\frac{n}{2}}} \frac{e^{i \theta n}}{2g}\hat{f}\left(\frac{n}{2g}\right)
\end{align*}
Therefore, for $U\in M_r(2g)$ with eigenangles $\theta_1,\dots,\theta_{2g}$,
\begin{align*}
\mathcal{D}_q(U,f) & = \sum_{\substack{s|r\\s\not=1}} \sum_{j=1}^{\frac{2g}{r-1}} \sum_{(k,s)=1} \sum_{n\in\Z} f\left( 2g \left(\frac{\theta_{\frac{r}{s}jk} +i\log(q)\left(\frac{1}{s}-\frac{1}{2}\right)}{2\pi}-n \right)\right) \\
& = \sum_{\substack{s|r\\s\not=1}} \sum_{j=1}^{\frac{2g}{r-1}} \sum_{(k,s)=1} \sum_{n=0}^{\frac{2g}{r-1}} \frac{q^{\frac{n}{s}}}{q^{\frac{n}{2}}} \frac{e^{i \theta_{\frac{r}{s}jk} n}}{2g}\hat{f}\left(\frac{n}{2g}\right) \\
& = \hat{f}(0) + \frac{1}{2g} \sum_{\substack{s|r\\s\not=1}} \sum_{n=1}^{\frac{2g}{r-1}}\frac{q^{\frac{n}{s}}}{q^{\frac{n}{2}}}\hat{f}\left(\frac{n}{2g}\right) \sum_{j=1}^{\frac{2g}{r-1}} \sum_{(k,s)=1} e^{i \theta_{\frac{r}{s}jk} n} \\
& = \hat{f}(0) + \frac{1}{2g} \sum_{\substack{s|r\\s\not=1}} \sum_{n=1}^{\frac{2g}{r-1}}\frac{q^{\frac{n}{s}}}{q^{\frac{n}{2}}}\hat{f}\left(\frac{n}{2g}\right) \Tr(U_s^n)
\end{align*}
where we recall that $U_s$ is defined to be $U$, restricted to the vector space with basis $\{e_{\frac{r}{s}jk}: j=1,\dots,\frac{2g}{r-1}, (k,s)=1 \}$, where the $e_i$ are the standard basis vectors. In particular, as $U$ runs over all $M_r(2g)$, then $U_s$ runs over all $M_{(s)}\left(\frac{\phi(s)2g}{r-1}\right)$. Hence,
\begin{align*}
\int_{M_r(2g)} \mathcal{D}_q(U,f) dU & = \hat{f}(0) + \frac{1}{2g} \sum_{\substack{s|r\\s\not=1}} \sum_{n=1}^{\frac{2g}{r-1}}\frac{q^{\frac{n}{s}}}{q^{\frac{n}{2}}}\hat{f}\left(\frac{n}{2g}\right) \int_{M_{(s)}\left(\frac{\phi(s)2g}{r-1}\right)} \Tr(U_s^n) dU \\
& = \hat{f}(0) - \frac{1}{2g} \sum_{\substack{s|r\\s\not=1}} \phi(s) \sum_{n=1}^{\frac{2g}{s(r-1)}}q^{n(1-\frac{s}{2})}\hat{f}\left(\frac{ns}{2g}\right)
\end{align*}

Combining everything, we then get
\begin{align*}
\langle \mathcal{D}(L_C,f) \rangle_{\Hh_r^{thin}(g)} & = \hat{f}(0) + \frac{1}{2g}\sum_{n=1}^{\frac{2g}{r-1}} \hat{f}\left(\frac{n}{2g}\right) \langle \Tr(\Theta_C^n) \rangle_{\Hh_r^{thin}(g)} \\
& = \hat{f}(0) - \frac{1}{2g} \sum_{n=1}^{\frac{2g}{r-1}} \hat{f}\left(\frac{n}{2g}\right)\left( \sum_{\substack{s|(n,r) \\ s\not=1 }} \phi(s)\frac{q^{n/s}}{q^{n/2}} + D_r(g,n)\right) + O\left(\frac{1}{g^2}\right) \\
& = \hat{f}(0) - \frac{1}{2g} \sum_{\substack{s|r \\ s\not=1}} \phi(s) \sum_{n=1}^{\frac{2g}{s(r-1)}} q^{n(1-\frac{s}{2})} \hat{f}\left(\frac{ns}{2g}\right) - \frac{1}{2g} \sum_{n=1}^{\frac{2g}{r-1}} \hat{f}\left(\frac{n}{2g}\right)D_r(g,n) + O\left(\frac{1}{g^2}\right) \\
& = \int_{M_r(2g)} \mathcal{D}_q(U,f) dU - \frac{1}{2g}\sum_{n=1}^{\frac{2g}{r-1}} \hat{f}\left(\frac{n}{2g}\right)D_r(g,n) + O\left(\frac{1}{g^2}\right)
\end{align*}

Finally, it remains to determine the contribution from $D_r(g,n)$. Recalling the definition of $D_r(g,n)$ from \eqref{D_r(g,n)}, we get that the contribution to the above will be
\begin{align*}
& \sum_{n=1}^{\frac{2g}{r-1}} \hat{f}\left(\frac{n}{2g}\right) \sum_{\deg(P)|n}\frac{\deg(P)}{q^{n/2}} \left(\left(r,\frac{n}{\deg(P)}\right)-1\right) \sum_{a=1}^{\lfloor \frac{d}{\deg(P)}\rfloor} \left(\frac{-\phi(r)}{q^{\deg(P)}}\right)^a \left(1-\frac{a\deg(P)}{d}\right)^{\phi(r)-1}\\
= & \sum_{m=1}^{\frac{2g}{r-1}} m \sum_{\deg(P)=m}   \sum_{a=1}^{\lfloor \frac{d}{m}\rfloor} \left(\frac{-\phi(r)}{q^{m}}\right)^a \left(1-\frac{am}{d}\right)^{\phi(r)-1} \sum_{k=1}^{\frac{2g}{m(r-1)}} \hat{f}\left(\frac{mk}{2g}\right) \frac{(r,k)-1}{q^{mk/2}}.
\end{align*}

Now, we have $\hat{f}(\frac{mk}{2g}) = \hat{f}(0) + O\left(\frac{mk}{2g}\right)$, and so
\begin{align*}
\sum_{k=1}^{\frac{2g}{m(r-1)}} \hat{f}\left(\frac{mk}{2g}\right) \frac{(r,k)-1}{q^{mk/2}} & = \hat{f}(0)\sum_{k=1}^{\frac{2g}{m(r-1)}} \frac{(r,k)-1}{q^{mk/2}} + O\left(\frac{m}{2g} \sum_{k=1}^{\frac{2g}{m(r-1)}}k \frac{(r,k)-1}{q^{mk/2}}\right) \\
& = \hat{f}(0) \sum_{\substack{s|r \\ s\not=1}} \phi(s)  \sum_{k=1}^{\frac{2g}{ms(r-1)}} \frac{1}{q^{msk/2}} +O\left(\frac{m}{gq^m}\right) \\
& = \hat{f}(0) \sum_{\substack{s|r \\ s\not=1}} \frac{\phi(s)}{q^{ms/2}-1} + O\left(\frac{m}{gq^m}\right)
\end{align*}

This error term contributes at most
$$\frac{1}{2g} \sum_{m=1}^\frac{2g}{r-1} \frac{m^2}{q^m} \sum_{\deg(P)=m}   \sum_{a=1}^{\lfloor \frac{d}{m}\rfloor} \left(\frac{\phi(r)}{q^{m}}\right)^a \left(1-\frac{am}{d}\right)^{\phi(r)-1} \ll \frac{1}{2g} \sum_{m=0}^\frac{2g}{r-1} \frac{m^2\pi_q(m)}{q^{2m}} \ll \frac{1}{2g}  $$

The remaining is thus
\begin{align*}
& \hat{f}(0) \sum_{\substack{s|r \\ s\not=1}}  \phi(s) \sum_{m=0}^{\frac{2g}{r-1}} \frac{m}{q^{ms/2}-1} \sum_{\deg(P)=m}   \sum_{a=1}^{\lfloor \frac{d}{m}\rfloor} \left(\frac{-\phi(r)}{q^{m}}\right)^a \left(1-\frac{am}{d}\right)^{\phi(r)-1}\\
= & \hat{f}(0)\sum_{\substack{s|r \\ s\not=1}} \phi(s) \sum_{P} \frac{\deg(P)}{|P|^{s/2}-1} \sum_{a=1}^{\lfloor \frac{d}{\deg(P)}\rfloor} \left(\frac{-\phi(r)}{|P|}\right)^a \left(1-\frac{a\deg(P)}{d}\right)^{\phi(r)-1} + O\left(\frac{1}{g}\right)
\end{align*}

\end{proof}

\bibliography{CyclicFirstDraft}
\bibliographystyle{amsplain}

\end{document}